\documentclass[11pt]{amsart}
\usepackage[centertags]{amsmath}

\usepackage{amsfonts}
\usepackage{amssymb}
\usepackage{amsthm}
\usepackage{newlfont}

\newcommand{\ber}{{\mathbb R}}

\newcommand{\emp}{\emptyset}
\newcommand{\cl}{c\ell}
\newcommand{\bigo}{\hbox{\bf O}}

\theoremstyle{plain}
\newtheorem{thm}{Theorem}[section]
\newtheorem{cor}[thm]{Corollary}

\newtheorem{prop}[thm]{Proposition}
\newtheorem{ex}[thm]{Example}

\theoremstyle{definition}
\newtheorem{defn}{Definition}[section]
\theoremstyle{remark}

\numberwithin{equation}{section}

\date{\today}

\begin{document}

\vspace{0.5in}

\title{Centrally Image partition Regularity near 0 }

\author{Tanushree Biswas}
\thanks{The first named author thanks University of Kalyani for support towards
her Ph.D programme.}
\address{Department of Mathematics, University of Kalyani, Kalyani-741235,
West Bengal, India}
\email{tanushreebiswas87@gmail.com}

\author{Dibyendu De}
\thanks{The second named author is partially supported by DST-PURSE programme.}
\address{Department of Mathematics, University of Kalyani, Kalyani-741235,
West Bengal, India}
\email{dibyendude@klyuniv.ac.in}

\author{Ram Krishna Paul}
\address{Department of Mathematics, Nagaland University, Lumami-798627, Nagaland,  India}
\email{rmkpaul@gmail.com}

\begin{abstract}
The notion of Image partition regularity near zero was first introduced
by De and Hindman. It was shown there that like image partition regularity
over $\mathbb{N}$ the main source of infinite image partition regular
matrices near zero are Milliken- Taylor matrices. But Milliken- Taylor matrices are far apart to
have images in central sets. In this regard the notion of centrally
image partition regularity was introduced. In the present paper we
propose the notion centrally partition regular matrices near zero for dense sub semigroup of $(\ber^+,+)$
which are different from centrally partition regular matrices unlike
finite cases. 

\end{abstract}

\maketitle

\section{Introduction}

It is well known that for finite matrices image partition regularity behaves well with respect to central subsets of the underlying semigroup (Central sets were introduced by Furstenberg \cite{refF} and enjoy very strong combinatorial properties \cite[Proposition 8.21]{refF}). But the situation becomes totally different for infinite image partition regular matrices. It was shown in \cite{refHLSc} that some of very interesting properties for finite image partition regularity could not be generalized for infinite image partition regular matrices. To handle these situations the notion of centrally  image partition regular matrices were introduced \cite{refHLSc}, while both these notions becomes identical for finite matrices. The same problem occurs in the setup of image partition regularity near zero over dense subsemigroup of $((0,\infty),+)$ which is stronger notion than image partition regularity. Again image partition regularity and image partition regularity near zero over dense subsemigroup of $((0,\infty),+)$ becomes identical for finite matrices. Also finite image partition regular matrices have images in any central sets as well as central set near zero for some nice dense subsemigroups of $((0,\infty),+)$. This situation motivates us to introduce the notion  of  {\em centrally image partition regular near zero over a dense subsemigroup of $((0,\infty),+)$  \/} which involve the notion of central sets near zero. Central sets near zero were introduced by Hindman and Leader \cite{refHL} and these sets also enjoy rich combinatorial structure like central sets.\\

 We shall present the notion central sets and central sets near zero after giving a brief description of algebraic structure of $\beta S_d$ for a discrete semigroup $(S, +)$. We take the points of $\beta S$ to be the ultrafilters on $S$, identifying the principal ultrafilters with the points of S and thus pretending that $S\subseteq\beta S$. Given $A\subseteq S$,
$$\cl A = \overline{A}= \{p\in\beta S : A\in p\}$$ is a basis for a topology on $\beta S$.
The operation $+$ on $S$ can be extended to the Stone-\v{C}ech compactification $\beta S$ of $S$ so that $(\beta S,+)$ is a compact right topological semigroup (meaning that for any $p \in \beta S$, the function $\rho_p : \beta S \rightarrow \beta S$ defined by $\rho_p(q) = q + p$ is continuous) with S contained in its topological center (meaning that for any $x \in S$, the function
$\lambda_x : \beta S \rightarrow \beta S$ defined by $\lambda_x(q) = x + q$ is continuous). Given $p,q\in \beta S$ and $A\subseteq S$, $A\in p + q$ if and only if $\{x\in S:-x+A\in q\}\in p$, where $-x+A=\{y\in S:x+y\in A\}$. \\

 A nonempty subset $I$ of a semigroup $(T,+)$ is called a \emph{left ideal of $S$} if $T+I\subset I$,  a \emph{right ideal} if $I+T\subset I$,
and a \emph{two sided ideal} (or simply an \emph{ideal}\/) if it is both a left and  right ideal.
A \emph{minimal left ideal} is the left ideal that does not contain any proper left ideal.
Similarly, we can define \emph{minimal right ideal} and \emph{smallest ideal}.

Any compact Hausdorff right topological semigroup $(T,+)$
has a smallest two sided ideal

$$\begin{array}{ccc}
    K(T) & = & \bigcup\{L:L \text{ is a minimal left ideal of } T\} \\
         & = & \,\,\,\,\,\bigcup\{R:R \text{ is a minimal right ideal of } T\}\\
  \end{array}$$

Given a minimal left ideal $L$ and a minimal right ideal
$R$, $L\cap R$ is a group, and in particular contains
an idempotent.  An idempotent in $K(T)$ is called
a {\it minimal\/} idempotent.  If $p$ and $q$ are idempotents in $T$
we write $p\leq q$ if and only if $p+q=q+p=p$. An idempotent
is minimal with respect to this relation if and only if it is a member of the smallest ideal.
See \cite{refHS} for an elementary introduction to the algebra of $\beta S$ and for any unfamiliar details.

\begin{defn}
Let $(S,+)$ be an infinite discrete semigroup. A set $C\subseteq S$ is central if and only if there is some minimal idempotent $p$ in $(\beta S, +)$ such that $C\in p$.
\end{defn}

We have been considering semigroups which are dense in $((0,\infty),+)$. Here ``dense" means with respect to the usual topology on $((0,\infty),+)$. When passing to the Stone-\v{C}ech compactification of such a semigroup $S$, we deal with $S_{d}$ which is the set $S$ with the discrete topology.

\begin{defn}
If $S$ is a dense subsemigroup  of $((0,\infty),+)$, then $0^+(S)=\{ p\in\beta S_d: (\forall\epsilon>0)((0,\epsilon)\cap S \in p)\}$.
\end{defn}

It is proved in \cite[Lemma 2.5]{refHL}, that $0^+(S)$ is a compact right topological subsemigroup of $(\beta S_d,+)$. It was also noted that $0^+(S)$ is disjoint from $K(\beta S_d)$ and hence gives some new information which are not available from $K(\beta S_d)$. Being compact right topological semigroup $0^+(S)$ contains minimal idempotents.  In \cite{refDH}, the authors applied the algebraic structure of $0^+(S)$ on their investigation of image partition regularity near zero of finite and infinite matrices. In \cite{refDPr} has been used algebraic structure of $0^+(\mathbb{R})$ to investigate image partition regularity of matrices with real entries from $\mathbb{R}$.\\

\begin{defn}
Let $S$ be a dense subsemigroup of $((0,\infty),+)$,
A set $C$ is central near $0$ if and only if there is some minimal idempotent $p$ in $0^+(S)$ such that $C\in p$.

\end{defn}

Next we present some well known characterizations of image partition regularity of matrices.

\begin{thm}\label{iprfinitech}
Let $u,v\in \mathbb{N}$ and let $M$ $u\times v$ matrix with entries from $\mathbb{Q}$. The following statements are equivalent.

\begin{itemize}
\item[(a)] M is image partition regular.

\item[(b)] For every central subset $C$ of $\mathbb{N}$, there exists $\vec x\in \mathbb{N}^{v}$ such that $M\vec x\in C^{u}$.

\item[(c)] For every central subset $C$ of $\mathbb{N}$, $\{\vec x\in \mathbb{N}^{v}$ : such that $M\vec x\in C^{u}\}$ is central in $\mathbb{N}^{v}$.

\item[(d)] For each $\vec r\in \mathbb{Q}^{v}\setminus \{\vec 0\}$ there exists $b\in \mathbb{Q}\setminus 0$ such that \\
    $$\left(\begin{array}{c}b\vec r\\
    M
 \end{array}\right)\,
$$
is image partition regular.

\item[(e)] For every central subset $C$ of $\mathbb{N}$, there exists $\vec x\in \mathbb{N}^{v}$ such that $\vec y=M\vec x\in C^{u}$, all entries of $\vec x$ are distinct, and for all $i,j\in \{1,2,\ldots,u\}$, if rows $i$ and $j$ of $M$ are unequal, then $y_{i}\neq y_{j}$.

\end{itemize}

\end{thm}

\begin{proof}

~\cite[Theorem 2.10]{refHLSb}

\end{proof}

In paper ~\cite {refHLSc}, the authors presented some contrasts between finite and infinite partition regular matrices and so showed that some of very interesting properties for finite image partition regularity could not be generalized for infinite image partition regular matrices. \\

It is interesting to observe that an
important  property is  an immediate
consequence of Theorem \ref{iprfinitech}(b), namely that if
$M$ and $N$ are finite
image partition regular matrices,  then the matrix
$$\left(\begin{array}{cc}M&\bigo\\ \bigo&N\\ \end{array}\right)$$
is also image partition regular. But this property does not hold good for infinite matrices as was shown in   ~\cite[Theorem 2.2]{refHLSc}.

\begin{thm}\label{mksepn}

Let $\vec b$ be a compressed sequence with entries from $\mathbb{N}$ such that $\vec b\neq (1)$. Let $M$ be a matrix whose rows are all rows $\vec a\in \mathbb{Q}^{\omega}$ with only finitely many nonzero entries such that $c(\vec a)=\vec b$. Let $N$ be the finite sums matrix.

\begin{itemize}

\item[(a)] The matrices $M$ and $N$ are image partition regular.

\item[(b)] There is a subset $C$ of $\mathbb{N}$ which is a member of every idempotent in $\beta \mathbb{N}$ (and is thus, in particular, central) such that for no $\vec x\in \mathbb{N}^{\omega}$ does one have $M\vec x \in C^{\omega}$.

\item[(c)] The matrix   $$\left(\begin{array}{cc}M&\bigo\\
\bigo&N
\end{array}\right)\,
$$
is not image partition regular.

\end{itemize}

\end{thm}

\begin{proof}

~\cite[Theorem 2.2]{refHLSc}

\end{proof}

 To overcome the above situation the following notion was introduced in ~\cite[Definition 2.7]{refHLSc}.

\begin{defn}\label{centralipr}

Let $M$ be an $\omega\times \omega$ matrix with entries from $\mathbb{Q}$. Then $M$ is {\em centrally image partition regular  \/} if and only if whenever $C$ is a central set in $\mathbb{N}$, there exists $\vec x\in \mathbb{N}^{\omega}$ such that $M\vec{x}\in C^{\omega}$.
\end{defn}

Note that the above definition ~\ref{centralipr} has a natural generalization for arbitrary  subsemigroup $S$ of $((0,\infty),+)$, and hence forth we will abbreviate this by CIPR/$S$. Motivation behind the introduction this new notion was that the principal good properties of finite image partition regular matrices could not be extended with respect to infinite image partition regular matrices.\\

It is easy to see that whenever $M$ and $N$ are {\em centrally image partition regular  \/} matrices over any subsemigroup $S$ of $((0,\infty),+)$, then so is
$$\left(\begin{array}{cc}M&\bigo\\
\bigo&N
\end{array}\right)\,.
$$\\

The above observation tells us that centrally image partition regular matrices are more natural candidate to generalize finite image partition regularity  in case of infinite matrices.\\
 In this course we introduce another natural candidate to generalize the properties of finite image partition regularity near zero in case of infinite matrices.

\begin{defn}

Let $M$ be an $\omega\times \omega$ matrix with entries from $\mathbb{Q}$ and let $S$ be a dense subsemigroup of $((0,\infty),+)$. Then $M$ is {\em centrally image partition regular near zero \/} if and only if whenever $C$ is a central set near zero in $S$, there exists $\vec x\in S^{\omega}$ such that $M\vec{x}\in C^{\omega}$.

\end{defn}

Hence forth for arbitrary  subsemigroup $S$ of $((0,\infty),+)$, we will abbreviate {\em centrally image partition regular near zero  \/} over $S$ by CIPR/$S_0$.

This is the simple fact that if $M$ and $N$ be two centrally image partition regular near zero matrices over a dense subsemigroup $S$ of $((0,\infty),+)$, then the diagonal sum $$\left(\begin{array}{cc}M&\bigo\\
\bigo&N
\end{array}\right)\,
$$
is also centrally image partition regular near zero matrix over a dense subsemigroup $S$ of $((0,\infty),+)$.

The following Examples show that there exists infinite matrices which are centrally image partition regular over $\mathbb{Q}^{+}$ but not centrally image partition regular near zero over $\mathbb{Q}^{+}$ and vice versa.

\begin{ex}\label{NnotRzp}

Let
$$M=\left(\begin{array}{ccccc}1&0&0&0&\ldots\\
2&1&0&0&\ldots\\
4&0&1&0&\ldots\\
8&0&0&1&\ldots\\
\vdots&\vdots&\vdots&\vdots&\ddots\end{array}\right)\,.
$$
Then $M$ is  CIPR/$\mathbb{Q}^{+}$  matrix but is not CIPR/$\mathbb{Q}^{+}_{0}$.

\end{ex}

\begin{proof}

 To see that $M$ is centrally image partition regular matrix, let $C$ be any central set in $\mathbb{Q}^{+}$
and pick a monochromatic sequence $\langle y_n\rangle_{n=0}^\infty$ in $C$
such that for each $n\in \mathbb{N}$, $y_n>2^ny_0$.  Let $x_0=y_0$ and for
each $n\in \mathbb{N}$, let $x_n=y_n-2^ny_0$.  Then $M\vec x=\vec y$.

Now $(0,1)\cap\mathbb{Q^{+}}$ is a central set near zero in $\mathbb{Q^{+}}$ and suppose one has $\vec x\in (\mathbb{Q}^{+})^\omega$ such that $\vec y=M\vec x\in
((0,1)\cap \mathbb{Q^{+}})^\omega$.
Then $x_0=y_0> 0$.  Pick $k\in \mathbb{N}$ such that $2^kx_0>1$.  Then
$y_k=2^kx_0+x_k> 1$, a contradiction.

\end{proof}

\begin{ex}\label{QzspnotD}  Let
$$M=\left(\begin{array}{cccccc}1&-1&0&0&0&\ldots\\
1/3&0&-1&0&0&\ldots\\
1/5&0&0&-1&0&\ldots\\
1/7&0&0&0&-1&\ldots\\
\vdots&\vdots&\vdots&\vdots&\vdots&\ddots\end{array}\right)\,.$$
Then $M$ is CIPR/$\mathbb{Q}^{+}_0$ but is not CIPR/$\mathbb{Q}^+$.
\end{ex}

\begin{proof} To see that $M$ is not CIPR/$\mathbb{Q}^{+}$, let $C$ be a central set in $\mathbb{Q}^{+}$ and we show that there
is no $\vec x\in (\mathbb{Q}^{+})^\omega$ such that $\vec y=M\vec x\in C^\omega$.
Indeed, suppose one has such $\vec x$ and pick $n\in\mathbb{N}$ such that
$x_0/(2n+1)\leqslant x_{0}$.  Then $y_n=x_0/(2n+1)-x_{n+1}$ is also bounded by $x_{0}$ in $\mathbb{Q}^+$.

To see that $M$ is CIPR/$\mathbb{Q}^{+}_0$ near zero
let $C$ be a central set near zero in $\mathbb{Q}^+$ such that $0\in \cl C$, and
pick a sequence $\langle y_n\rangle_{n=0}^\infty$ in $C$ which converges to $0$.
We may also assume that for each $n$, $y_n<1/(2n+1)$.  Let $x_0=1$
and for $n\in\mathbb{N}$, let $x_n=1/(2n-1)-y_{n-1}$.  Then $M\vec x=\vec y\in C^\omega$.
\end{proof}



In ~\cite{refHLSc}, we have seen that finite image partition regularity matrices hold some interesting properties but not infinite image partition regular matrices. In this paper we show this behaviour is also true for the notion of image partition regularity near zero. This is why we introduce the notion centrally image partition regularity near zero.
Now in section $2$ of this paper, we first prove that for two infinite image partition regular matrices near zero, i.e. $M$ and $N$, over $\mathbb{D}^{+}$ the diagonal sum $$\left(\begin{array}{cc}M&\bigo\\
\bigo&N
\end{array}\right)\,
$$
is not image partition regular near zero over $\mathbb{D}^{+}$. But we show that infinite image partition regular near zero matrices can be extended by finite ones. Also we show in proposition that how new types of centrally infinite image partition regular matrices near zero are constructed from old one.\\
In section $3$, we prove that a special type of infinite image partition regular matrices (i.e. segmented image partition regular matrices) are also centrally image partition regular near zero.

\section{Centrally image partition regularity of matrices near zero}

In Theorem ~\ref{mksepn} we have found two infinite image partition regular matrices $M$ and $N$ over $\mathbb{N}$ while the diagonal sum $$\left(\begin{array}{cc}M&\bigo\\
\bigo&N
\end{array}\right)\,
$$
is not image partition regular matrix over $\mathbb{N}$. But the central tool to prove the above Theorem is Milliken-Taylor separating theorem
~\cite[Theorem 3.2]{refDHLL}. Recently in ~\cite{refWd}, Milliken-Taylor separating theorem has been proved for dyadic rational numbers which employ to prove the following  generalization of ~\ref{mksepn}. First we recall some Definitions from ~\cite{refWd}.

\begin{defn}

The set of dyadic rational numbers is given by\\ $\mathbb{D}=\{\frac{m}{2^{t}}$ : $m\in \mathbb{Z}$ and $t\in \omega\}$.

\end{defn}

We will be considering $\mathbb{D}^{+}$, the set of positive numbers contained in $\mathbb{D}$.

\begin{defn}

Let $x\in \mathbb{D}^{+}$. The {\em support\/} of $x$, denoted supp$(x)$, is the unique finite nonempty subset of $\mathbb{Z}$ such that $x=\sum_{t\in supp(x)}2^{t}$.

\end{defn}

\begin{defn}

Given a binary number, an \emph{even} $0$-\emph{block} is the occurrence of a positive even total of consecutive zeros between two consecutive ones.

\end{defn}

For $x\in \mathbb{D}^{+}$, define the \emph{start} of $x$ as the position of the first $1$ appearing in $x$ moving from left to right and the \emph{end} as the position of the last $1$. The formal definition is the following.

\begin{defn}
Let $x\in\mathbb{D}^+$. Then $x=\sum_{t\in\text{supp}(x)}2^t$ where $\text{supp}(x)\in\mathcal{P}_f(\mathbb{Z})$. Define the \emph{start} of $x$ as the max supp$(x)$ and the \emph{end} as the min supp$(x)$.

\end{defn}

Now we present the following Proposition from ~\cite[Proposition 2.12]{refWd} that play the key role to prove the following Theorem \ref{NotCentrally}.

\begin{prop}\label{separtmt}

Let $\varphi(z)$ be the number of even $0$-blocks between the start and end of $z$ for any $z\in \mathbb{D}\cap (0,2)$. For $i\in \{0,1,2\}$, let $C_{i}=\{c\in \mathbb{D}\cap (0,2)$ : $\varphi(c)\equiv i$ mod $3\}$. Then $\{C_{0}, C_{1}, C_{2}\}$ is a partition of $\mathbb{D}\cap (0,2)$ such that no $C_{i}$ contains $MT(\langle 1 \rangle,\langle x_i\rangle_{i=1}^\infty)\cup
MT(\langle 1,2 \rangle,\langle y_i\rangle_{i=1}^\infty)$ for any sequences
$\langle x_i\rangle_{i=1}^\infty$ and $\langle
y_i\rangle_{i=1}^\infty$ in  $\mathbb{D}\cap (0,2)$.
\end{prop}

\begin{proof}

~\cite[Proposition 2.12]{refWd}

\end{proof}

\begin{thm}\label{NotCentrally}

Let $M$ be finite sum matrix and $N$ be the Milliken-Taylor matrix determined by  compressed sequence $\langle 1,2\rangle$. Then

\begin{itemize}

\item[(a)] The matrices $M$ and $N$ are image partition regular near zero over $\mathbb{D}^{+}$.


\item[(b)] The matrix   $$\left(\begin{array}{cc}M&\bigo\\
\bigo&N
\end{array}\right)\,
$$
is not image partition regular near zero over $\mathbb{D}^{+}$.
\item[(c)] The matrix $N$ is not centrally image partition regular near zero over $\mathbb{D}^{+}$.

\end{itemize}

\end{thm}

\begin{proof}

Statement [a] follows from ~\cite[Theorem 5.7]{refDH}.

From ~\ref{separtmt} the matrix is $$\left(\begin{array}{cc}M&\bigo\\
\bigo&N
\end{array}\right)\,
$$ not image partition regular near zero over $\mathbb{D}^{+}$.

Again, since the matrix $$\left(\begin{array}{cc}M&\bigo\\
\bigo&N
\end{array}\right)\,
$$
is not image partition regular near zero over $\mathbb{D}^{+}$. Therefore $N$ is not CIPR/$\mathbb{D}^{+}_0$ as $M$ has its image in every central set near zero.

Let $N$ is centrally image partition regular near zero. Again $M$ is centrally image partition regular near zero follows from ~\cite[Theorem 3.1]{refHL}. Then the matrix $$\left(\begin{array}{cc}M&\bigo\\
\bigo&N
\end{array}\right)\,
$$
is  centrally image partition regular near 0 and hence also image partition regular near 0. But this is a contradiction. Therefore $N$ is not centrally image partition regular near zero over $\mathbb{D}^{+}$.

%

\end{proof}

Now we show that infinite image partition regular near zero matrices can be extended by finite ones.

\begin{thm}\label{extensionipr0}

Let $M$ be a finite image partition regular matrix over $\mathbb{N}$ and $N$ be an infinite image partition regular
near zero matrix over any dense subsemigroup $S$ of $((0,\infty),+)$. Then $$\left(\begin{array}{cc}M&\bigo\\
\bigo&N
\end{array}\right)\,
$$
is image partition regular near zero.

\end{thm}

\begin{proof}

Let $S$ be $r$-colored by $\varphi$ as $S=\bigcup_{i=1}^r C_i$ and $\epsilon>0$. By a standard compactness argument (see \cite[Section 5.5]{refHS} )
there exists $k\in\mathbb{N}$ such that whenever $\{1,2,\cdots,k\}=\bigcup_{i=1}^r D_i$ there exists
$\vec{x}\in\{1,2,\cdots,k\}^v$ and $i\in\{1,2,\cdots,r\}$ such that $M\vec{x}\in (D_i)^u$. Pick $z\in S \cap (0,\epsilon/k)$.

Now color $S$ with $r^k$ colors via $\psi$ as $S=\bigcup_{i=1}^{r^k} F_i$, where $\psi(x)=\psi(y)$ if and only if for all $t\in\{1,2,\cdots,k\}$, $\varphi(tx)=\varphi(ty)$. Choose $\vec{y}\in S^\omega$ such that the entries of $N\vec{y}$ are in $F_i\cap (0,z)$ for some $i\in \{1,2,\cdots,r^{k}\}$.  Pick an entry $a$ of $N\vec{y}$ and for each $i\in\{1,2,\cdots,r\}$ let us set $D_i=\{t\in\{1,2,\cdots,k\}:ta\in C_i\}$. Then $\{1,2,\cdots,k\}=\bigcup_{i=1}^r D_i$. Note that since $a\in (0,z)$, $ta\in (0,\epsilon)$ for all $t\in\{1,2,\cdots,k\}$. If we express this coloring as $\gamma:\{1,2,\cdots,k\}\rightarrow \{1,2,\cdots,r\}$ then in fact $\gamma(p)=\varphi(ap)$.
So there exists $\vec{u}\in\{1,2,\cdots,k\}^v$ and $i\in\{1,2,\cdots,r\}$ such that $M\vec{u}\in (D_i)^u$ so that $a(M\vec{u})\in (C_i)^u$. Now $a(M\vec{u})=M(a\vec{u})$. Put $\vec{x}=a\vec{u}$. Then $M\vec{x}\in (C_i\cap (0,\epsilon))^u$. Choose an entry $i$ of $M\vec{u}$ and let $j=\gamma(i)$.

Let $\vec{z}=\left(\begin{array}{cc}a\vec{u}\\i\vec{y}\end{array}\right)$. We claim that for any row $\vec{w}$ of $\left(\begin{array}{cc}M & O\\O & N\end{array}\right)$, $\varphi(\vec{w}\cdot\vec{z})=j$. To observe this first assume that $\vec{w}$ is a row of $\left(\begin{array}{cc}M & O\end{array}\right)$, so that $\vec{w}={\vec{s}}^\frown\vec{0}$, where $\vec{s}$ is a row of $M$. Then $\vec{w}\cdot\vec{z}=\vec{s}\cdot (a\vec{u})=a(\vec{s}\cdot\vec{u})$. Therefore $\varphi(\vec{w}\cdot\vec{z})=\varphi(a(\vec{s}\cdot\vec{u}))=\gamma(\vec{s}\cdot\vec{u})=j$.

Next assume that $\vec{w}$ is a row of $\left(\begin{array}{cc}O & N\end{array}\right)$, so that $\vec{w}= \vec{0}^\frown\vec{s}$ where $\vec{s}$ is a row of $N$. Then $\vec{w}\cdot\vec{z}=i(\vec{s}\cdot\vec{y})$. Now $\psi(\vec{s}\cdot\vec{y})=\psi(a)$. So $\varphi(i(\vec{s}\cdot\vec{y}))=\varphi(ia)=\gamma(i)=j$.

\end{proof}

Now  we shall show how new type of infinite centrally image partition regular matrix can be constructed from old one (that is extended up to infinite order i.e. here up to $\omega$).\\
Henceforth unless otherwise stated $S$ will be considered as dense subsemigroup of  $((0,\infty),+)$ for which $cS$ is $central^{*}$ near zero for every $c\in \mathbb{N}$.\\
We now present the following theorem and corollary to prove the following proposition \ref{infinitecentralnear0}.

\begin{thm}\label{cardidempotent}

Let $S$ be a subsemigroup of $((0,\infty),+)$. Let $p\in
K(0^{+}(S))$, let $C\in p$, and let $R$ be the minimal right
ideal of $0^{+}(S)$ to which $p$ belongs. Then there are at
least countably infinitely many idempotents in $K(0^{+}(S))\cap
R\cap \overline{C}$.

\end{thm}

\begin{proof}

~\cite[Theorem 2.3]{refDHS09}

\end{proof}

\begin{cor}\label{infinitecentral}

Let $S$ be a dense subsemigroup of $((0,\infty),+)$ and let $C$ be a central set near zero. Then there exists a sequence $\langle C_{n} \rangle_{n=1}^{\infty}$ of pairwise disjoint central sets near zero in $S$ with $\bigcup _{n=1}^{\infty}C_{n}\subseteq C$.

\end{cor}

\begin{proof}

By the above Theorem  ~\ref{cardidempotent},  there are at
least countably infinitely many idempotents in $\overline{C}$.  hence contains an infinite strongly discrete subset. (Alternatively, there are two minimal idempotents in $\overline{C}$ so that $C$ can be split into two central sets near zero, $C_{1}$ and $D_{1}$. Then $D_{1}$ can be split into two central sets near zero, $C_{2}$ and $D_{2}$, and so on.)

\end{proof}

\begin{prop}\label{infinitecentralnear0}

For each $n\in \mathbb{N}$, let $M_{n}$ be a {\em centrally image partition regular near zero  \/} matrix. Then the matrix
$$M=\left(\begin{array}{cccc}M_{1}&0&0&\ldots\\
0&M_{2}&0&\ldots\\
0&0&M_{3}&\ldots\\
\vdots&\vdots&\vdots&\ddots
\end{array}\right)\,.
$$
is also {\em centrally image partition regular near zero  \/}.

\end{prop}

\begin{proof}

Let $C$ be a central sets near zero and choose by the above Corollary  ~\ref{infinitecentral} a sequence $\langle C_{n} \rangle_{n=1}^{\infty}$ of pairwise disjoint central sets near zero in $S$ with $\bigcup _{n=1}^{\infty}C_{n}\subseteq C$. For each $n\in \mathbb{N}$ choose $\vec x^{(n)}\in S^{\omega}$ such that $\vec y^{(n)}=M_{n}\vec x^{(n)}\in C_{n}^{\omega}$. Let  $$\vec z =\left(\begin{array}{c}\vec x^{(1)}\\
\vec x^{(2)}\\
\vdots
\end{array}\right)\,.
$$
Then all entries of $M\vec z$ are in $C$.

\end{proof}

\section{Some infinite Centrally image partition regularity of matrices near zero}
We now present a class of image partition regular matrices which are called segmented image partition regular matrices introduced in \cite{refHLi}. And we show that these class of matrices are also infinite centrally image partition regular matrices.

\begin{defn}
Let $M$ be an $\omega\times\omega$ matrix with entries from $\mathbb{Q}$. Then $M$ is a segmented image partition regular matrix if and only if
\begin{enumerate}
\item no row of $M$ is row is $\vec0$;
\item for each $i\in \omega$, $\{j\in \omega : a_{i,j}\neq\emp\}$ is finite; and
\item there is an increasing sequence $\langle\alpha_{n}\rangle_{n=0}^{\infty}$ in $\omega$ such that $\alpha_{0}=0$ and for each $n\in \omega$,\\
$\{\langle a_{i,\alpha_{n}},a_{i,\alpha_{n}+1},a_{i,\alpha_{n}+2},\ldots,a_{i,\alpha_{n+1}-1}\rangle : i\in \omega\}\setminus \{\vec 0\}$\\
 is empty or is the set of rows of a finite image partition regular matrix.

\end{enumerate}
\end{defn}
If each of these finite image partition regular matrices is a first entries matrix, then $M$ is a segmented first entries matrix. If also the first nonzero entry of each $\langle a_{i,\alpha_{n}},a_{i,\alpha_{n}+1},a_{i,\alpha_{n}+2},\ldots,a_{i,\alpha_{n+1}-1}\rangle$, if any, is 1, then $M$ is a monic segmented first entries matrix.

\begin{thm}
Let $S$ be a dense subsemigroup of $((0,\infty),+)$ and let $M$ be a segmented image partition regular matrix with $\omega$. Then $M$ is  \emph{centrally image partition regular near zero}.

\end{thm}

\begin{proof}

Let $\vec c_{0}, \vec c_{1}, \vec c_{2},\ldots$ denote the columns of $M$. Let $\langle\alpha_{n}\rangle_{n=0}^{\infty}$ be as in the definition of a segmented image partition regular matrix. For each $n\in \omega$, let $M_{n}$ be the matrix whose columns are $\vec c_{\alpha_{n}},\vec c_{\alpha_{n}+1},\ldots,\vec c_{\alpha_{n+1}-1}$. Then the set of non-zero rows of $M_{n}$ is finite and, if nonempty, is the set of rows of a finite image partition regular matrix. Let $B_{n} = (M_{0}$  $M_{1}\ldots M_{n})$.

Now by ~\cite[Lemma 2.5]{refHL}
 $0^{+}(S)$ is a compact right topological semigroup so that we can choose an minimal idempotent $p\in 0^{+}(S)$. Let $C\subseteq S$
such that $C\in p$. Let $C^{*}=\{x\in C : -x+C\in p\}$. Then $C^{*}\in p$ and, for every $x\in C^{*}$, $-x+C^{*}\in p$ by ~\cite[Lemma 4.14]{refHS}.\\
Now the set of non-zero rows of $M_{n}$ is finite and, if nonempty, is the set of rows of a finite image partition regular matrix over $\mathbb{N}$ and hence by ~\cite[Theorem 2.3]{refDH} $IPR/S_{0}$. Then by ~\cite[Theorem 4.10]{refDH} , we can choose $\vec x^{(0)}\in S^{\alpha_{1}-\alpha_{0}}$ such that, if $\vec y=M_{0}\vec x^{(0)}$, then $y_{i}\in C^{*}$ for every $i\in \omega$ for which the $i^{th}$ row of $M_{0}$ is non-zero.\\
We now make the inductive assumption that, for some $m\in \omega$, we have chosen $\vec x^{(0)},\vec x^{(1)},\ldots,\vec x^{(1)}$ such that $\vec x^{(i)}\in S^{\alpha_{i+1}-\alpha_{i}}$ for every $i\in \{0,1,2,\ldots,m\}$, and, if

$$\vec y=B_{m}\left(\begin{array}{c}\vec x^{(0)}\\
\vec x^{(1)}\\ . \\.\\.\\\vec x^{(m)}\end{array}\right),$$

then $y_{j}\in C^{*}$ for every $j\in \omega$ for which the $j^{th}$ row of $B_{m}$ is non-zero.\\
Let $D=\{j\in \omega$ : row $j$ of $B_{m+1}$ is not $\vec 0\}$ and note that for each $j\in \omega, -y_{j}+C^{*}\in p$. (Either $y_{j}=0$ or $y_{j}\in C^{*}$) By  ~\cite[Theorem 4.10]{refDH} we can choose $\vec x^{(m+1)}\in S^{\alpha_{m+2}-\alpha_{m+1}}$ such that, if $\vec z=M_{m+1}\vec x^{(m+1)}$, then $z_{j}\in \bigcap _{t\in D}(-y_{t}+C^{*})$ for every $j\in D$.\\

Thus we can choose an infinite sequence $\langle \vec x^{(i)} \rangle_{i\in \omega}$ such that, for every $i\in \omega$, $\vec x^{(i)}\in S^{\alpha_{i+1}-\alpha_{i}}$, and, if

$$\vec y=B_{i}\left(\begin{array}{c}\vec x^{(0)}\\
\vec x^{(1)}\\ . \\.\\.\\\vec x^{(i)}\end{array}\right),$$

then $y_{j}\in C^{*}$ for every $j\in \omega$ for which the $j^{th}$ row of $B_{i}$ is non-zero.\\

Let $$\vec x=\left(\begin{array}{c}\vec x^{(0)}\\
\vec x^{(1)}\\ \vec x^{(2)}\\\vdots \end{array}\right)$$

and let $\vec y=M\vec x$. We note that, for every $j\in \omega$, there exists $m\in \omega$ such that $y_{j}$ is the $j^{th}$ entry of

$$B_{i}\left(\begin{array}{c}\vec x^{(0)}\\
\vec x^{(1)}\\ . \\.\\.\\\vec x^{(i)}\end{array}\right)$$

whenever $i>m$. Thus all the entries of $\vec y$ are in $C^{*}$.

\end{proof}



\bibliographystyle{amsplain}

\end{document}